\documentclass[12pt]{amsart}

\usepackage{geometry}
\usepackage[all]{xy}

\usepackage{amsmath, amsthm}
\usepackage{amssymb}
\usepackage{amscd}
\usepackage{latexsym}
\usepackage{epsfig}
\usepackage{graphics}
\usepackage{amsfonts}
\usepackage{psfrag}
\usepackage{graphicx}

\def\N {{\mathbb{N}}}
\def\Q {{\mathbb{Q}}}
\def\Z {{\mathbb{Z}}}
\def\R {{\mathbb{R}}}
\def\F {{\mathcal{F}}}
\def\I {{\mathcal{I}}}
\def\J {{\mathcal{J}}}
\def\C {{\mathbb{C}}}
\def\cdots {{\cdot\cdot\cdot}}
\newtheorem{theorem}{Theorem}[section]
\newtheorem{lemma}[theorem]{Lemma}
\newtheorem{proposition}[theorem]{Proposition}
\newtheorem{corollary}[theorem]{Corollary}
\newtheorem{definition}[theorem]{Definition}

\theoremstyle{definition}
\newtheorem{example}[theorem]{Example}
\newtheorem{remark}[theorem]{Remark}

\numberwithin{equation}{section}

\begin{document}

\title{Cohomology rings of good contact toric manifolds}

\author{Shisen Luo}

\address{Department of Mathematics, Cornell University,
Ithaca, NY 14853-4201, USA}

\email{{\tt ssluo@math.cornell.edu}}


\begin{abstract}
A good contact toric manifold $M$ is determined by its moment cone
$C$. We compute the equivariant cohomology ring with $\Z$
coefficient of $M$ in terms of the combinatorial data of $C$. Then
under a smoothness criterion on the cone $C$, we compute the
singular cohomology ring with $\Z$ coefficient of $M$ in terms of
the combinatorial data of $C$.
\end{abstract}

\date{\today}

\maketitle \tableofcontents
\section{\bf Introduction}\label{sec:introduction}

\subsection{\it Good contact toric manifolds}

 A {\em contact toric manifold} of dimension
$2n-1$ is a compact connected $(2n-1)$-dimensional contact manifold
equipped with an effective Hamiltonian action of an $n$-torus $T^n$.
One can find an introduction to Hamiltonian actions on contact
mainfolds and definitions of related notions in \cite{L:contact
toric}. We will not use these concepts in any essential way in this
paper. Motivated by the work of Banyaga and Molino
\cite{BnM1},\cite{BnM2},\cite{Banyaga}, and of Boyer and Galicki
\cite{BoyG}, Lerman gave a full-classification theorem of contact
toric manifolds in \cite{L:contact toric}. According to Lerman's
theorem, when $n\geq 3$ and the torus action is not free, the
contact toric manifolds are classified by their {\it moment cones},
which by definition are the union of $\{origin\}$ and the moment map
image of their symplectizations. These moment cones are all {\it
good cones}.
\begin{definition}\label{def:good cone}
 A {\bf good cone} in $\R^n$ is a rational polyhedral cone given by
\begin{equation}
C= \bigcap_{i=1}^{m}\{x\in \R^n: \langle x, v_{i}\rangle \geq 0\},
\end{equation}
where $F_{i}=\{x\in \R^n: \langle x,v_{i} \rangle=0\}$ is a facet of
$C$ and $v_{i}\in \Z^n$ is the inward-pointing primitive normal
vector to $F_{i}$. In addition this cone must satisfy:

(i)For $0<l<n$, each codimension $l$ face $F$ of $C$ is contained in
exactly $l$ facets:
\begin{equation*}
F=F_{i_{1}}\cap F_{i_{2}}\cap\cdots\cap F_{i_{l}};
\end{equation*}
and

(ii) The $\Z$-module generated by $v_{i_{1}},...,v_{i_{l}}$ is a
direct summand of $\Z^n$ of rank $l$.
\end{definition}

\begin{definition}\label{def:good contact toric manifolds}
A contact toric manifold $M$ is called a {\bf good contact toric
manifold} if $dim M > 3$ and the moment cone of $M$ is a strictly
convex good cone.
\end{definition}

This explains the title of this paper.

\begin{remark}
The requirement that the moment cone being strictly convex is
equivalent to asking $M$ to be of {\em Reeb type}.
\end{remark}

\begin{remark}
When the moment cone is not strictly convex, the construction of
contact toric manifold with the given moment cone can be found in
the proof of Proposition 4.7 in \cite{LS:integrable torus actions}.
They are topologically $T^k\times S^{k+2l-1}$.
\end{remark}

\subsection {\it Equivariant cohomology}

Assume $G$ is a Lie group and $M$ is a topological space with a
$G$-action. Let $EG$ be a contractible topological space with a free
$G$-action. Then $G$ acts freely on $EG\times M$ diagonally. The
quotient space of this action, which we will denote by $(EG\times
M)/G$ and $EG\times_{G}M$ interchangeably, is called the {\em Borel
construction} of the $G$-space $M$. The homotopy type of
$EG\times_{G}M$ is independent of the choice of $EG$. Notations such
as $M_{G}$ and $EG\times^{G}M$ are also used in literature to denote
the Borel construction of $M$. The {\em equivariant cohomology} of
$M$ is defined as the ordinary cohomology of $EG\times_{G}M$ and is
denoted $H_{G}^{*}(M)$.

The equivariant cohomology of a point with trivial torus action is
of particular interest to us in this paper. The Borel construction
of a point $ET^m\times_{T^m}pt =ET^m/T^m$ is the classifying space
of $T^m$, denoted by $BT^m$. If we use $(S^{\infty})^{m}$ as the
model for $ET^m$, then $BT^m$ is the product of $m$ copies of $\C
P^{\infty}$. The equivariant cohomology of a point is thus a
polynomial ring with $m$ variables, i.e,
\begin{equation}
H_{T^m}^{*}(pt;\Z) = \Z[x_{1},x_{2},...,x_{m}],
\end{equation}
where $x_{i}\in H^{2}_{T^m}(pt;\Z)$.

 The variable $x_{i}$ is the first Chern class of the fiber bundle
\begin{equation}
ET^m\times_{T^m}\C\rightarrow ET^m\times_{T^m}pt,
\end{equation}
where in the total space, $T^m$ acts on $\C$ with weight
$-(0,0,...,1,...,0)$, where the $1$ is in the $i^{th}$ position.

The generators for the various cohomology rings in this paper will
always be the images of these $x_{i}$'s under certain maps.

It is not hard to see that equivariant cohomology is an equivariant
homeomorphism invariant. In this paper, we are concerned with
equivariant and ordinary cohomology rings. Since both of these are
equivariant homeomorphism invariants, two $T^n$-spaces that are
equivariantly homeomorphic are considered as identical spaces.
Because of this, we may exploit an idea of Davis and Januskiwicz in
\cite{DJ}, where they define a topological counterpart for toric
manifolds and compute the corresponding cohomology rings.

\subsection{\it Outline of the paper}

In Section~\ref{sec:basic con}, we imitate the construction given in
\cite{DJ} to define several topological spaces with torus actions,
using combinatorial methods. These spaces will be shown in later
sections to be equivariantly homeomorphic to toric symplectic cones,
good contact toric manifolds and symplectic toric orbifolds
respectively.

In Section~\ref{sec:equi contact}, exploiting the techniques in
\cite{DJ} with some modification, we compute the equivariant
cohomology of a good contact toric manifold. The main theorem is
Theorem~\ref{theorem:equivariant cohomology contact}.

In Section~\ref{sec:symplectic orbifold}, we collect some facts
about symplectic toric manifolds and setup some notation for
Section~\ref{sec:cohomology contact}.

In Section~\ref{sec:cohomology contact}, we compute the ordinary
cohomology of a good contact toric manifold. The work in this
section is not carried out in full generality. We must first assume
$C_{0}$, the cone minus the origin, is contained in the upper half
space
\begin{equation}
U\R^n = \{(x_{1},...,x_{n})\in \R^n : x_{n}>0\}.
\end{equation}
We will show that this assumption will not cause any loss of
generality.

The intersection of $C$ with the hyperplane
\begin{equation}
H = \{(x_{1},...,x_{n})\in \R^n : x_{n}=1\}
\end{equation}
is a simple rational convex polytope. We denote it by $P$. Assume
$v_{i}=(v_{i1},v_{i2},...,v_{in})$ is the primitive inward-pointing
normal vector to  $F_{i}$, the $i^{th}$ facet of the cone. We need
to impose the following criterion on $C$.

{\bf Smoothness Criterion:} The polytope $P$ is a Delzant polytope
in $H$, and the vector $(v_{i1},v_{i2},...,v_{i,n-1})$ is primitive
in $\Z^{n-1}$.

  This hypothesis is where the
generality is lost. We call this a {\em smoothness criterion}, since
a Delzant polytope is by definition a smooth simple rational convex
polytope. We make this hypothesis so we do not need to deal with
orbi-bundles over orbifolds.

Under this smoothness criterion, we show that $M$, a good contact
toric manifold, is a principle $S^1$-bundle over a symplectic toric
manifold $N$. Then using the Gysin sequence of the $S^1$-bundle, we
compute the ordinary cohomology group of $M$ and also show how to
take the product of any two even degree cohomology classes, and the
product of one even degree cohomology class and one odd degree
cohomology class. This is Theorem ~\ref{thm:contact cohomology}.
Then by relating the Euler class of the $S^1$-bundle to the
symplectic form on $N$ and using the Hard Lefschetz Theorem, we show
that the odd degree cohomology of $M$ vanishes in degree lower than
half, and hence show the product of any two odd degree cohomology
classes of $M$ is zero for dimension reasons. This is Theorem
~\ref{thm:vanishing odd}.

Finally in the Appendix, we give several equivalent descriptions of
the generators of the cohomology ring of a symplectic toric
manifold.
\subsection{\it Relation with other work}

Some of the geometric computations we describe in this paper have
been computed in a more algebraic fashion by other authors. The
virtue of our computation is that it is more explicit and geometric.
Moreover, the geometry allows us to use the Hard Lefschetz Theorem
on symplectic toric manifold to deduce Theorem 5.15. The consequent
vanishing of certain Betti numbers is much more obscure in the
existing algebraic description of the cohomology ring. The integral
equivariant cohomology ring of a general smooth toric variety, which
includes the symplectization of good contact toric manifolds, was
identified with the Stanley-Reisner ring earlier by Franz
\cite[Sec.\ 3]{Franz} with a different proof from our proof of
Theorem~\ref{theorem:equivariant cohomology contact}. In Theorem 1.2
of the same paper of Franz, he got an expression of the integral
ordinary cohomology of a smooth toric variety in terms of Tor
modules. The ordinary cohomology ring of the quotient of a
moment-angle complex, which includes the case we consider in Section
5, was expressed also using Tor modules in \cite[Thm.7.37]{BP}. To
relate the description in Theorem ~\ref{thm:contact cohomology} in
this paper to the results of [Fr] and \cite{BP} uses an algebraic
argument that will be explained in [LMM].

 \

{\bf A remark on notations}: If a group $G$ acts on two spaces $X$
and $Y$, we use $X\times_{G} Y$ and $(X\times Y)/G$ interchangeably
to denote the quotient space of $X\times Y$ under the diagonal $G$
action. And we will use $[x,y]$ to denote the element of
$X\times_{G}Y$ that is the equivalence class of $(x,y)\in X\times
Y$, so $[gx, gy]=[x,y]$, for $g\in G$.

{\bf Acknowledgement}: I would like to thank Tara Holm, Miguel
Abreu, Chi-Kwong Fok, Allen Hatcher, Allen Knutson, Tomoo Matsumura,
Frank Moore and Reyer Sjamaar for some useful conversations.

\section{\bf Basic definitions,constructions and examples}\label{sec:basic con}

A $2n$ dimensional {\bf symplectic toric manifold} is a compact
connected symplectic manifold equipped with an effective Hamiltonian
action of an $n$-torus $T^n$. Delzant showed in \cite{Delzant} that
these geometric objects are classified by their moment image, which
is a simple rational smooth polytope, called a {\bf Delzant
polytope} in the symplectic literature. More information about
symplectic toric manifolds may be found in Chapter 28 of
\cite{CdS:book}.

The analogous result for contact toric manifolds was given by Lerman
in \cite{L:contact toric}, where he showed that a large class of
contact toric manifolds are classified by {\bf good cones}, defined
in Definition ~\ref{def:good cone}.

 The full
classification theorem of contact toric manifolds is Theorem 2.18 in
\cite{L:contact toric}. The terminology {\bf good contact toric
manifold}, which is defined in Definition~\ref{def:good contact
toric manifolds}, used in this paper belongs to the case(4) in
Lerman's classification theorem. It is actually a proper subcase,
since here we also require the moment cone to be strictly convex.

A {\bf toric symplectic cone} is the symplectization of a contact
toric manifold. Toric symplectic cones and contact toric manifolds
are in one-to-one correspondence. There is more information about
symplectic cones in \cite{L:geodesic flows} and \cite{AM:contact
homology}. We will review the method to obtain toric symplectic
cones and contact toric manifolds from strictly convex good cones in
Section~\ref{sec:equi contact}.

The Delzant polytope, which is the moment image of symplectic toric
manifold, is in fact the orbit space of the $T^n$ action, and the
moment map is just the point to orbit map. Using the ideas of
\cite{DJ}, with simple combinatorial methods we can construct
manifolds that are $T^n$-equivariantly homeomorphic to symplectic
toric manifolds. The constructions we describe below are a variation
and generalization of those in \cite{DJ}.

We let $P$ (or $C_{0}$) be a simple convex polytope (or a strictly
convex good cone minus the origin) in $\R^n$. The set of facets of
$P$ (or $C_{0}$) is denoted $\F$, and we write $F_{i}$ for the
$i^{th}$ facet. A {\bf characteristic map} is a map
\begin{equation*}
\lambda: \F\rightarrow \Z^l
 \end{equation*}
 where $\Z^l$ is the integral lattice in $\R^l$. Here $l$
and $n$ are not necessarily equal. Denote $\lambda(F_{i})$ by
$\lambda_{i}$.

Every finite subset $U$ of $\Z^l$ determines a subgroup of $T^l$
generated by
\begin{equation*}\{(e^{u_{1}\theta},e^{u_{2}\theta},...,e^{u_{l}\theta}):\theta\in
\R,(u_{1},...,u_{l})\in U\}\end{equation*} It is in fact a closed
subgroup. For every point $p$ in $P$(or $C_{0}$), denote by
$S^{\lambda}_{p}$ the subgroup of $T^l$ determined by vectors
\begin{equation*}\{\lambda_{i}:p\in F_{i}\}.\end{equation*}
Define an equivalence relation $\Delta$ on $T^{l}\times P$ by:
\begin{equation}
(g,p)\Delta (h,q) \Leftrightarrow  p=q, and\ g^{-1}h\in
S^{\lambda}_{p}.
\end{equation}
 We say $\Delta$ is the equivalence relation
determined by $\lambda$. Let $\Delta_{p}$ denote $S^{\lambda}_{p}$,
and $P^{\lambda}$ denote the quotient space $(T^{l}\times
P)/\Delta$. With the quotient topology and the $T^{l}$ action on the
first coordinate by multiplication, $P^{\lambda}$ is in the category
of $T^{l}$-topological spaces.

We now give three fundamental examples which will be heavily used
throughout the paper.

\begin{example}\label{ex:symplectic cone combina}
  Let
  \begin{equation}
  C= \bigcap_{i=1}^{m}\{x\in \R^n: \langle x, v_{i}\rangle \geq 0\}
\end{equation}
be a strictly convex good cone in $\R^n$. The vector $v_{i}$ is the
inward-pointing primitive normal vector to $F_{i}$. We assume the
equations are {\bf minimal}; that is, removing any one of the
equations will give a different set. We will always assume that the
equations used to define convex polytopes or cones are minimal. We
let $C_{0}=C\backslash
 \{origin\}$. We set $l=n$ and define a characteristic map $\lambda$ by $\lambda(F_{i})=v_{i}$.
 We denote by $\Delta$ the equivalence relation on $T^n\times C_{0}$ determined
 by $\lambda$.
 Then $C_{0}^{\lambda}= (T^n\times C_{0})/\Delta$, with $T^n$ acting on the first coordinate by multiplication, is
$T^n$-equivariantly homeomorphic to a toric symplectic cone. This
claim will be proved in Section~\ref{sec:equi contact}.
\end{example}

\begin{example}\label{ex:contact combina}
 For $C$ as in the previous example, it is a cone over a simple convex polytope $P$.
Notice that the facets of $P$ are in one-to-one correspondence with
those of $C_{0}$. We let $\tilde{F_{i}}$ denote the facet of $P$
contained in $F_{i}$. We may define a characteristic map as in
Example~\ref{ex:symplectic cone combina} by sending $\tilde{F_{i}}$
to $v_{i}$. By abuse of notation,we still call this map $\lambda$
and denote by $\Delta$ the equivalence relation on $(T^n\times P)$
determined by $\lambda$. We emphasize that $\lambda_{i}=
\lambda(\tilde{F_{i}})$ is the normal vector to $F_{i}$, not to
$\tilde{F_{i}}$ in $P$.  Then $P^{\lambda}= (T^n\times P)/\Delta$,
with $T^n$ acting on the first coordinate by multiplication, is
$T^n$-equivariantly homeomorphic to the contact toric manifold $M$
associated to $C$. The space $C_{0}^{\lambda}$ in the previous
example is the symplectization of $M=P^{\lambda}$. These claims will
be proved in Section~\ref{sec:equi contact}.
\end{example}

The last example is related to the polytope $P$ itself, without
reference to the cone $C$.
\begin{example}\label{ex:sym orb combinatorics}
Let $P$ be a Delzant polytope, i.e. a convex simple rational smooth
polytope, in $\R^{k}$ given by
\begin{equation}\label{eq:simple polytope}
P = \bigcap_{i=1}^{m}\{x\in \R^k: \langle x, \tilde{v_{i}}\rangle
\geq \eta_{i}\}.\end{equation} The equations are again assumed to be
minimal, as in Example~\ref{ex:symplectic cone combina}. We denote
the $i^{th}$ facet by $\tilde{F_{i}}$. Then $\tilde{v_{i}}$ is the
inward-pointing primitive normal vector to the facet
$\tilde{F_{i}}$. We set $l=k$ and define a characteristic map
$\tilde{\lambda}$ by setting $\tilde{\lambda}(\tilde{F_{i}})=
\tilde{v_{i}}$. Denote by $\tilde{\Delta}$ the equivalence relation
on $T^k\times P$ determined by $\tilde{\lambda}$. Then
$P^{\tilde{\lambda}}=(T^k\times P)/\tilde{\Delta}$, with $T^k$
acting on the first coordinate by multiplication, is
$T^{k}$-equivariantly homeomorphic to the symplectic toric manifold
associated to $P$. This will be restated in
Section~\ref{sec:symplectic orbifold}.
\end{example}
Notice the subtle difference between $P^{\lambda}$ and
$P^{\tilde{\lambda}}$. Their relation is crucial in this paper.

\section{\bf Equivariant cohomology of good contact toric manifolds}\label{sec:equi contact}
In this section, we imitate the computation of the equivariant
cohomology of a symplectic toric manifold given in \cite{DJ} to
compute the equivariant cohomology of a good contact toric manifold.

\begin{remark}
In their paper \cite{DJ}, Davis and Januszkiewicz defined a class of
manifolds equipped with torus action which they called {\em toric
manifolds}, now called {\em quasitoric manifolds}, and computed
their equivariant and ordinary cohomology rings. The cohomology ring
of symplectic toric manifolds is a particular example of the
cohomology ring of a quasitoric manifold. Quasitoric manifolds are a
strictly larger class than symplectic toric manifolds. More
information may be found in \cite{GP:quasitoric}. Constructions in
\cite{DJ} were later generalized to orbifolds in \cite{PS:quasitoric
orbifolds}.
\end{remark}

We begin with a brief review of the construction of a toric
symplectic cone from a strictly convex good cone given in Lemma 6.4
in \cite{L:contact toric}. Let the cone be as in
Example~\ref{ex:symplectic cone combina}:
 \begin{equation}
  C= \bigcap_{i=1}^{m}\{x\in \R^n: \langle x, v_{i}\rangle \geq 0\}
\end{equation}
Define a map
\begin{equation}\label{eq:pi_Z contact}
\pi_{\Z}: \Z^m\rightarrow \Z^n
\end{equation}
by sending the $i^{th}$ standard basis vector $e_{i}$ of $\Z^m$ to
$v_{i}$. This induces a map
\begin{equation}
\pi_{\R}: \R^m\rightarrow \R^n
\end{equation}
by tensoring with $\R$. We then get a map
\begin{equation}\label{eq:pi_T contact}
\pi_{T}: T^m=\R^m/\Z^m \rightarrow T^k=\R^n/\Z^n
\end{equation}
The subscripts $\Z,\R,T$ may be omitted when it will not cause
confusion.

Let $K=ker (\pi_{T})$. Then we have a short exact sequence of groups
\begin{equation}\label{eq:KT^mT^n}
\xymatrix{
  1 \ar[r]^{} & K \ar[r]^{i} & T^m \ar[r]^{\pi_{T}} & T^n \ar[r]^{} & 1 .   }
\end{equation}
This induces maps between Lie algebras
\begin{equation}
\xymatrix{
  0 \ar[r]^{} & \mathfrak{k} \ar[r]^{i_{*}} & \mathfrak{t}^m \ar[r]^{\pi_{*}} & \mathfrak{t}^n \ar[r]^{} & 0 ,  }
\end{equation}
with dual maps between the duals of the Lie algebras
\begin{equation}\xymatrix{0 & \mathfrak{k}^*\ar[l]^{} &
(\mathfrak{t}^m)^* \ar[l]_{i^*}& (\mathfrak{t}^n)^*\ar[l]_{\pi^*}& 0
.\ar[l]}
\end{equation}

Let $u:\C^m\rightarrow (\mathfrak{t}^m)^*$ be defined by
\begin{equation}\label{eq:std moment contact}
u(z_{1},...,z_{m})=(|z_{1}|^2,...,|z_{m}|^2).
\end{equation}
Lerman showed in \cite{L:contact toric} that
\begin{equation} S=((i^{*}\circ u)^{-1}(0)\backslash \{0\})/K
\end{equation} is the toric symplectic cone associated to
$C$.

The standard action of $T^m$ on $\C^m$ restricts to a $T^m$-action
on $(i^{*}\circ u)^{-1}(0)\backslash \{0\}$, and hence a $K$-action
on $(i^{*}\circ u)^{-1}(0)\backslash \{0\}$. It induces a $T^n\cong
T^m/K$ action on $S=((i^{*}\circ u)^{-1}(0)\backslash \{0\})/K$.

In other words, the action of $t_{n}\in T^n$ on $S$ is induced by
the standard action of any element in $\pi^{-1}(t_{n})$ on
$(i^{*}\circ u)^{-1}(0)\backslash \{0\}$. This action is Hamiltonian
with moment map
\begin{eqnarray}\label{eq:moment toric }
\nu: & S& \rightarrow  (\mathfrak{t}^n)^{*}\\
&[z_{1},...,z_{m}]& \mapsto (\pi^{*})^{-1}(u(z_{1},...,z_{m})),
\end{eqnarray}
using that $\pi^{*}$ is injective. The image of $\nu$ is
$C_{0}=C\backslash\{0\}$ according to \cite{L:contact toric}. It is
not hard to see from the construction that $C_{0}$ is the orbit
space of the $T^n$ action on $S$.

Now we are ready to prove the claim made in
Example~\ref{ex:symplectic cone combina} as follows.
\begin{proposition}
With the same notation as in Example~\ref{ex:symplectic cone
combina}, $C_{0}^{\lambda}$ is $T^n$-equivariantly homeomorphic to
$S$, the symplectic toric cone associated with $C$.
\end{proposition}
\begin{proof}

If we restrict the domain of $u$ defined in \eqref{eq:std moment
contact} to
\begin{equation}(\R^m)^{+}=\{(z_{1},...,z_{m})\in \C^m: z_{i}\in \R_{\geq
0},\forall i\},\end{equation} it is a homeomorphism onto its image.
Call this restriction map $u_{0}$.

 Now define a map \begin{equation}\gamma_{0}: T^n\times C_{0}\rightarrow
 S\end{equation} by:
 \[(t_{n},p)\mapsto t_{n}.[u_{0}^{-1}(\pi^{*}(p))]\]

We note that $\pi^{*}(p)\in (\R^m)^{+}$ because of the defining
equations for $C$. Thus, $u_{0}^{-1}(\pi^{*}(p))$ is well-defined,
and it is straightforward to check that
\begin{equation*}u_{0}^{-1}(\pi^{*}(p))\in (i^{*}\circ u)^{-1}(0)\backslash
\{0\}.\end{equation*} We let $[u_{0}^{-1}(\pi^{*}(p))]\in S$ be the
equivalence class of $u_{0}^{-1}(\pi^{*}(p))$. Finally, $t_{n}\in
T^n$ acts on $S$ as noted earlier.

 This map $\gamma_{0}$ is a continuous. It is also surjective because $C_{0}$ is the orbit space of the $T^n$-action on
 $S$.

 According to Lemma~\ref{lemma:stabilizer}, which will be stated and proved right after this proposition, $\gamma_{0}$ induces a bijective map
 \begin{equation}\label{eq:gamma}
 \gamma: (T^n\times C_{0})/\Delta\rightarrow
 S\end{equation}
 The inverse of this map is also continuous since $\gamma$ is an
 open map. Finally, $\gamma$ is obviously $T^n$-equivariant.
\end{proof}

\begin{lemma}\label{lemma:stabilizer}
 The stabilizer of $[u_{0}^{-1}(\pi^{*}(p))]\in S$ is
 $\Delta_{p}\leq T^n$.
 \end{lemma}

\begin{proof}
For simplicity and without loss of generality, we may assume that
the facets that contain $p$ are exactly $F_{1},...,F_{j}$. Thus, the
coordinates of $\pi^{*}(p)$ that are zero are exactly the first $j$
coordinates. Consequently, the stabilizer of
$[u_{0}^{-1}(\pi^{*}(p))]$ must be
$$
\{
\pi_{T}(e^{i\theta_{1}},e^{i\theta_{2}},...,e^{i\theta_{j}},1,1,...,1):
\theta_{1},...,\theta_{j}\in \R\}.$$ This is precisely $\Delta_{p}$.
\end{proof}

A few more lines will prove the claims made in
Example~\ref{ex:contact combina}.

\begin{proposition}\label{prop:contact combina}
Let $M$ denote the good contact toric manifold associated to $C$.
With the same notation as in Example~\ref{ex:contact combina},
$P^{\lambda}$ is $T^n$-equivariantly homeomorphic to $M$.
\end{proposition}
\begin{proof}
By definition, the symplectic cone $S$ is the symplectization of
$M$. Topologically, $S= M\times \R$. The map $\nu$ defined in
\eqref{eq:moment toric } is proportional on the second coordinate.

For each $q\in M$, there is a unique $x=x(q)\in \R$, such that
$$\nu(q,x(q))\in P .$$
Notice that $P^{\lambda}=(T^n\times P)/\Delta$ is a subset of
$(T^n\times C_{0})/\Delta$, so $(T^n\times P)/\Delta$ is
$T^n$-equivariantly homeomorphic to its image under $\gamma$ as
defined in \eqref{eq:gamma}. This is precisely the pre-image of $P$
under $\nu$, which is
$$\{(q,x(q))\in S=M\times \R : q\in M\}.$$ This is clearly $T^n$-equivariantly
homeomorphic to $M$.
\end{proof}

We now compute the $T^n$ equivariant cohomology of $M$, or
equivalently, of $P^{\lambda}$. Recall that $C$ is a cone over $P$.
We continue to let $\lambda$ denote the characteristic map defined
in Example~\ref{ex:contact combina}. Suppose
$v_{i}=(v_{i1},v_{i2},...,v_{in})$. Let $\tilde{\F}$ be the set of
facets of $P$, and $m=|\tilde{\F}|$, the number of facets. Define a
characteristic map
\begin{equation}\label{eq:mu}
 \mu
:\tilde{\F}\rightarrow \Z^m
\end{equation}
by sending $\tilde{F_{i}}$ to $e_{i}$, the $i^{th}$ standard basis
vector of $\Z^m$. Denote by $\Omega$ the equivalence relation on
$T^m\times P$ determined by $\mu$. This space $P^{\mu}=(T^m\times
P)/\Omega$ was first defined in \cite{DJ}.

\begin{lemma}\label{lemma:fiberK}
$P^{\mu}$ is a fiber bundle over $P^{\lambda}$ with fiber $K$, where
$K$ was defined in \eqref{eq:KT^mT^n}.
\end{lemma}
\begin{proof} The group $K$ acts naturally on $P^{\mu}=(T^m\times
P)/\Omega$ by multiplication on the first coordinate. Since $C$ is a
good cone, $K\cap \Omega_{p}=1$ for every $p\in P$.
 Thus
the $K$ action on $P^{\mu}$ is free. The orbit space is exactly
$P^{\lambda}=(T^n\times P)/\Delta$. The projection from the total
space to the orbit space is given by the natural map
\begin{eqnarray*}\pi':& P^{\mu}=(T^m\times P)/\Omega&\rightarrow
P^{\lambda}=(T^n\times P)/\Delta\\     &[t_{m},p]&\mapsto
[\pi_{T}(t_{m}),p].
\end{eqnarray*}
\end{proof}

The torus $K$ may be disconnected. Denote by $K_{0}$ the connected
component of $K$ containing identity.

\begin{lemma}\label{lemma:torus retraction}
 There exists a group homomorphism $r: T^m\rightarrow K_{0}$,
 satisfying
$r\circ i|_{K_{0}}=id_{K_{0}}$, where $id_{K_{0}}$ denotes the
identity map on $K_{0}$ and $i$ is the inclusion of $K$ into $T^m$
as defined in \eqref{eq:KT^mT^n}.
 \end{lemma}
\begin{proof}Consider the
maps\[\xymatrix{
  0 \ar[r]^{} & ker(\pi_{\Z}) \ar[r]^{j} & \Z^m \ar[r]^{\pi_{\Z}} & \Z^n
  },\]
  where $\pi_{\Z}$ is defined in \eqref{eq:pi_Z contact} and $j$ is
  the inclusion map.
  As a subgroup of the free abelian group $\Z^m$, $ker(\pi_{\Z})$ is
  also a free abelian group. So with suitably chosen basis, the
  map $j$ looks like an inclusion \[b_{1}\Z\oplus b_{2}\Z\oplus \cdot\cdot\cdot \oplus b_{k}\Z
  \hookrightarrow \Z\oplus\Z\oplus \cdot\cdot\cdot \Z\oplus
  \Z^{m-k}.\]
  But notice that for every $x\in \Z^m$ and every $t\in \Z\backslash\{0\}$, \[x\in ker(\pi_{\Z})\Leftrightarrow tx\in
  ker(\pi_{\Z}).\]
  So all of the $b_{i}$ must be $1$. This means
  $ker(\pi_{\Z})$ is a direct summand of $\Z^m$, so there is a group homomorphism \begin{equation}
  r_{0}:\Z^m\rightarrow
  ker(\pi_{\Z}),\end{equation}
  such that
  \[r_{0}\circ j = id_{ker(\pi_{\Z})}.\]
   This
  $r_{0}$ induces a map
\begin{equation}\label{eq:torus retraction}
r: T^m\rightarrow K_{0}
\end{equation}
that satisfies the requirement of the lemma.
  \end{proof}

Using Lemma~\ref{lemma:torus retraction}, we can define an action of
$T^m$ on $EK_{0}\times ET^n$ by first sending $t_{m}\in T^m$ to
$(r(t_{m}), \pi(t_{m}))\in K_{0}\times T^n$, then using diagonal
action of this element on $EK_{0}\times ET^n$.

\begin{lemma}\label{lemma:EK0 bundle}
 The space $(EK_{0}\times ET^n)\times_{T^m}((T^m\times P)/\Omega)$ is a
fiber bundle over\\  $ET^n\times_{T^n}((T^n\times P)/\Delta)$ with
fiber $EK_{0}$. Thus these two spaces are homotopy equivalent.
\end{lemma}
\begin{proof}
Define the projection map by
 \[[x,y,[t,p]]\mapsto
[y,[\pi(t),p]],\] where $x\in EK_{0}$, $y\in ET^n$, $t\in T^m$,
$p\in P$. The square brackets $[\ ]$ are used whenever there is a
equivalence relation involved.

It is easily shown that the map is well-defined. As for the fiber,
pick any point $[y_{0},[t_{n},p]]$ in the base space, with $y_{0}\in
ET^n$, $t_{n}\in T^n$, $p\in P$.
 Assume $[x,y,[t,p]]$
is in the pre-image, then \[[y,[\pi(t),p]]=[y_{0},[t_{n},p]].\]
 So there exists
$s\in T^n$ such that \[sy=y_{0},\ \ \  \mbox{and}\]
\[[s\cdot\pi(t),p]=[t_{n},p].\]
Since $\pi_{T}$ is surjective, so there is an $s'\in T^m$ such that
$\pi(s')=s$.
 Then \[[x,y,[t,p]]=[r(s')x,\pi(s')y,[s't,p]]=[r(s')x,
y_{0}, [s't,p]],\ \mbox{and}\]
 \[ \pi'[(s't),p]=[\pi(s')\pi(t),p]= [s\cdot\pi(t),p]= [t_{n},p],\]
 where $\pi'$ is defined in Lemma~\ref{lemma:fiberK}.

This means when considering the fiber over $[y_{0},[t_{n},p]]$, we
only need to use representatives of the form $[x,y_{0},[t,p]],$
where $y_{0}$ is fixed, and $\pi'[t,p]= [t_{n},p]$.

By Lemma~\ref{lemma:fiberK}, we know that the set of points
$[t,p]\in P^{\lambda}$ that satisfy $\pi'[t,p]=[t_{n},p]$ is
homeomorphic to $K$.
  The elements in
$T^m$ that fix $y_{0}$ are $K$.
 So the fiber over $[y,[t_{n},p]]$
is homeomorphic to $(EK_{0}\times K)/K$, which is homeomorphic to
$EK_{0}$ via the map \[[x,k]\mapsto r(k)^{-1}x,\] noticing that $K$
acts on $EK_{0}$ by first mapping $k\in K$ to $r(k)\in K_{0}$, then
applying $r(k)$ to $EK_{0}$.
\end{proof}

Notice that the cohomology of $ET^n\times_{T^n}((T^n\times
P)/\Delta)$ is exactly what we want to compute: $H_{T^n}^{*}(M;\Z)$.

Davis and Januskiewicz computed
$H_{T^m}^{*}(P^\mu;\Z)=H^{*}(ET^m\times_{T^m}((T^m\times
P)/\Omega);\Z)$ in \cite{DJ}. More details can be found in
\cite[434-436]{DJ}. So the remaining task is to compare the two
spaces\[(EK_{0}\times ET^n)\times_{T^m}((T^m\times P)/\Omega)\]
 and
\[ET^m\times_{T^m}((T^m\times P)/\Omega).\]

\begin{lemma}\label{lemma:AG Property}
The two spaces
\begin{equation}
(EK_{0}\times ET^n)\times_{T^m}((T^m\times P)/\Omega)
\end{equation}
and
\begin{equation}
ET^m\times_{T^m}((T^m\times P)/\Omega)
\end{equation}
are homotopy equivalent.
\end{lemma}

\begin{proof}
Let $W= EK_{0}\times ET^n\times ET^m$. An element $t_{m}\in T^m$
acts on $W$ by the diagonal action of $(r(t_{m}), \pi(t_{m}),
t_{m})$.

The projection
\begin{equation}\label{eq:pp1}
p_{1}: W\times_{T^m}(T^m\times P)/\Omega \rightarrow
ET^m\times_{T^m}(T^m\times P)/\Omega
\end{equation} is a fiber
bundle with fiber $EK_{0}\times ET^n$, which is contractible, so
\begin{equation}\label{eq:p1}
W\times_{T^m}(T^m\times P)/\Omega\stackrel{h.e.}{\sim}
ET^m\times_{T^m}(T^m\times P)/\Omega , \end{equation} where
$\stackrel{h.e.}{\sim}$ stands for `homotopy equivalent'.

Suppose $t_{m}\in T^m$. If $t_{m}\not\in K$, then $\pi(t_{m})\neq
1$, so $t_{m}$ acts on $EK_{0}\times ET^n$ with no fixed-points. If
$t_{m}\not\in \Omega_{p}$ for any $p\in P$, then $t_{m}$ acts on
$(T^m\times P)/\Omega$ with no fixed-points.  Since $K\cap
\Omega_{p}=1$ for all $p\in P$, the diagonal action of $T^m$  on
$(EK_{0}\times ET^n)\times (T^m\times P)/\Omega$ is free.

So the projection
\begin{equation}\label{eq:pp2}
p_{2}: W\times_{T^m}(T^m\times P)/\Omega \rightarrow (EK_{0}\times
ET^n)\times_{T^m} (T^m\times P)/\Omega
\end{equation}
is a fiber bundle with fiber $ET^m$. So
\begin{equation}\label{eq:p2}
W\times_{T^m}(T^m\times P)/\Omega \stackrel{h.e.}{\sim}
(EK_{0}\times ET^n)\times_{T^m} (T^m\times P)/\Omega.
\end{equation}
Combining \eqref{eq:p1} and \eqref{eq:p2} completes the proof.
\end{proof}

\begin{definition}\label{def:I}
Define $\I$ to be the ideal of $\Z[x_{i},x_{2},...,x_{m}]$ or
$\Q[x_{i},x_{2},...,x_{m}]$ generated by monomials
\begin{equation*}\left\{x_{i_{1}}x_{i_{2}}\cdots x_{i_{l}}:
\bigcap_{j=1}^{l}\widetilde{F_{i_{j}}}=\emptyset\right\}.\end{equation*}
The coefficients used will be clear in context.
\end{definition}

\begin{theorem}\label{theorem:equivariant cohomology contact}
 If $M$ is the contact toric manifold associated with the strictly convex good cone \[C= \bigcap_{i=1}^{m}\{x\in
(t^{n})^{*}=\R^n: \langle x, v_{i}\rangle \geq 0\}\] then
\[H_{T^n}^{*}(M;\Z)\simeq \Z[x_{1},...,x_{m}]/\I,\]
where $x_{i}\in H^{2}_{T^n}(M;\Z)$.
\end{theorem}
\begin{proof} According to Lemma~\ref{lemma:EK0 bundle},\[H_{T^n}^{*}(M;\Z)=H^{*}(ET^n\times_{T^n} P^{\lambda},\Z)=
H^{*}((EK_{0}\times ET^n)\times_{T^m}P^{\mu};\Z)\]
 According to Lemma~\ref{lemma:AG Property}, they are equal
to \[H_{T^m}^{*}(P^{\mu};\Z).\] Then
Theorem~\ref{theorem:equivariant cohomology contact} follows from
Theorem 4.8 in \cite{DJ}.
\end{proof}

\begin{corollary}\label{cor:H^1 of M}
For a good contact toric manifold $M$, we have $H^{1}(M;\Z)=0$.
\end{corollary}
\begin{proof}
Consider the spectral sequence of the fiber bundle
\begin{equation}
M\hookrightarrow ET^n\times_{T^n}M \rightarrow BT^n
\end{equation}
If $E_{2}^{0,1}=H^{1}(M;\Z)$ has torsion, it will be in the kernel
of
\begin{equation}
d_{2}: E_{2}^{0,1}\rightarrow E_{2}^{2,0}
\end{equation}
since $E_{2}^{2,0}=H^{2}(BT^n;\Z)=\Z^n$ is free. So the torsion part
will live to $E_{\infty}^{0,1}$. Thus $H^{1}(ET^n\times_{T^n}M;\Z)$
will have torsion and contradicts Theorem~\ref{theorem:equivariant
cohomology contact}, from which it easily follows that
$H^{1}(ET^n\times_{T^n}M;\Z)=0$.  So $H^{1}(M;\Z)$ is torsion-free.

Theorem 1.1 in \cite{L:homotopy group} tells us that $\pi_{1}(M)$ is
a finite group. So $H_{1}(M;\Z)$ is finite, and $H^{1}(M;\Z)$ has no
free part.

Thus, we may conclude that $H^{1}(M;\Z)=0$.
\end{proof}

\section{\bf Cohomology and equivariant cohomology of symplectic toric
manifolds}\label{sec:symplectic orbifold} In this section, we recall
some old construction and facts about symplectic manifolds. This
will help to set up notations for Section \ref{sec:cohomology
contact}. In Section \ref{sec:cohomology contact}, we will take
$k=n-1$.

 First we will briefly recall the classical Delzant construction of a symplectic toric
 maniifold.
Let the convex Delzant polytope be as in \eqref{eq:simple polytope}:
\begin{equation}
P= \bigcap_{i=1}^{m}\{x\in \R^k: \langle x, \tilde{v_{i}}\rangle
\geq \eta_{i}\}.\end{equation}
 Then define a map
\begin{equation}\label{eq:pi_Z}
\pi_{\Z}: \Z^m\rightarrow \Z^k
\end{equation}
by sending the $i^{th}$ standard basis vector $e_{i}$ of $\Z^m$ to
$\tilde{v_{i}}$. This induces a map
\begin{equation}
\pi_{\R}: \R^m\rightarrow \R^k
\end{equation}
by tensoring with $\R$. We then get a map
\begin{equation}\label{eq:pi_T}
\pi_{T}: T^m=\R^m/\Z^m \rightarrow T^k=\R^k/\Z^k
\end{equation}
The subscripts $\Z,\R,T$ will be omitted sometimes when it will not
cause confusion.

Letting $K=ker(\pi_{T})$, we have a short exact sequence of groups
\begin{equation}\label{eq:KT^mT^k}
\xymatrix{
  1 \ar[r]^{} & K \ar[r]^{i} & T^m \ar[r]^{\pi_{T}} & T^k \ar[r]^{} & 1.   }
\end{equation}
This induces maps on Lie algebras
\begin{equation}
\xymatrix{
  0 \ar[r]^{} & \mathfrak{k} \ar[r]^{i_{*}} & \mathfrak{t}^m \ar[r]^{\pi_{*}} & \mathfrak{t}^k \ar[r]^{} & 0 ,  }
\end{equation}
and maps on the duals of Lie algebras
\begin{equation}\xymatrix{0 & \mathfrak{k}^*\ar[l]^{} &
(\mathfrak{t}^m)^* \ar[l]_{i^*}& (\mathfrak{t}^k)^*\ar[l]_{\pi^*}&
0.\ar[l]}
\end{equation}

Let $u:\C^m\rightarrow (\mathfrak{t}^m)^*$ be defined by
\begin{equation}\label{eq:std moment}
u(z_{1},...,z_{m})=(|z_{1}|^2,...,|z_{m}|^2)
\end{equation}
Then
\begin{equation}\label{eq:toric manifold as quotient}
N=((i^*\circ u)^{-1}(i^{*}(-\eta)))/K
\end{equation}
is the symplectic toric manifold associated to the polytope $P$,
where $\eta=(\eta_{1},...,\eta_{m})\in \R^m= (\mathfrak{t}^m)^*$.

The standard action of $T^m$ on $\C^m$ restricts to a $T^m$-action
on $(i^*\circ u)^{-1}(i^{*}(-\eta))$, and hence a $K$-action on
$(i^*\circ u)^{-1}(i^{*}(-\eta))$. It induces a $T^k\cong T^m/K$
action on $N=((i^*\circ u)^{-1}(i^{*}(-\eta)))/K$.

In other words, the action of $t_{k}\in T^k$ on $N$ is induced by
the standard action of any element in $\pi^{-1}(t_{k})$ on
$(i^*\circ u)^{-1}(i^{*}(-\eta))$. This action is Hamiltonian with
moment map
\begin{eqnarray}\label{eq:moment map}
\nu: & N& \rightarrow  (\mathfrak{t}^k)^{*}\\
&[z_{1},...,z_{m}]& \mapsto (\pi^{*})^{-1}(u(z_{1},...,z_{m})+\eta),
\end{eqnarray}
using that $\pi^{*}$ is injective. The image of $\nu$ is $P$. It is
easy to check from the construction that $P$ is the orbit space of
the $T^k$-action on $N$.

Following the same line as Proposition~\ref{prop:contact combina},
we have the following.

\begin{proposition}\label{prop: DJ construction is OK for sym orb}
With the same notation as in Example~\ref{ex:sym orb combinatorics},
$P^{\tilde{\lambda}}$ is $T^{k}$-equivariantly homeomorphic to $N$,
the symplectic manifold associated to $P$.
\end{proposition}

The following diagram was used in \cite{DJ} to compute
$H^{*}(N;\Z)=H^{*}((T^k\times P)/\tilde{\Delta};\Z)$.

\begin{equation}\label{eq:diagram sym orb} \xymatrix{
N\ar[r]^(0.4){f_1} & ET^k\times_{T^k}N & ET^m\times_{T^m}((T^m\times
P)/\Omega) \ar[l]_(0.6){f_2} \ar[r]^(0.6){f_3} &
ET^m\times_{T^m}pt}.
\end{equation}

In this diagram, $ET^m$ is taken to be $EK\times ET^k$. Since $K$ is
connected, there is a map $r: T^m\rightarrow K$, such that $r|_{K}$
is identity map. A group element  $t_m \in T^m$  acts on
$ET^m=EK\times ET^k$ via the diagonal action of $(r(t_m),\pi(t_m))$.

The map $f_1$ is inclusion of fiber, $f_3$ is the projection onto
the first factor. The map $f_2$ is given by $[(x,y),[t_m, p]]\mapsto
[y,[\pi(t_m),p]]$, where $x\in EK, y\in ET^k, t_m\in T^m, p\in P$.

\begin{theorem}\cite{DJ}\label{thm:symp orbi equi cohomology}
Taking cohomology of diagram~\eqref{eq:diagram sym orb} allows us to
compute the equivariant cohomology of the symplectic toric manifold
\begin{equation}
H_{T^k}^{*}(N;\Z)=H^{*}(ET^k\times_{T^k}(T^k\times
P)/\tilde{\Delta});\Z)= \Z[x_{1},x_{2},...,x_{m}]/\I\ ,
\end{equation}
where $x_{i}\in H_{T^k}^{2}(N;\Z)$, $1\leq i\leq m$, are images of
generators of $H^{*}(ET^m\times_{T^m}pt, \Z)$. And the ideal $\I$
was defined in Definition~\ref{def:I}.
\end{theorem}

\begin{definition}
Suppose $\tilde{v_{j}}=(v_{j1},v_{j2},...,v_{jk})$, for $1\leq j\leq
m$. Let
\begin{equation}J_{i}=\sum_{j=1}^{m}v_{ji}x_{j},
\end{equation}
for $1\leq i\leq k$. We define
\begin{equation}
\J=\langle J_{1},J_{2},...,J_{k}\rangle
\end{equation}
to be the ideal in $\Q[x_{1},x_{2},...,x_{m}]$ generated by the
linear terms $\{J_{1},J_{2},...,J_{k}\}$.
\end{definition}

\begin{theorem}\cite{DJ}\label{thm:symp orbi cohomology}
Taking cohomology of diagram~\eqref{eq:diagram sym orb} allows us to
compute the singular cohomology of symplectic toric manifold.
\begin{equation}
H^{*}(N;\Z)=\Z[x_{1},x_{2},...,x_{m}]/\langle \I,\J\rangle ,
\end{equation}
where $x_{i}\in H^{2}(N;\Z)$, for $1\leq i\leq m$.
\end{theorem}

If we denote by $T^r$ the subgroup of $T^k=(S^1)^k$ that is the last
$r$ copies of $S^1$, then the $T^k$ action on $N$ restricts to a
$T^r$ action on $N$. We may compute the $T^r$-equivariant cohomology
of $N$ as an easy corollary of the Theorem~\ref{thm:symp orbi equi
cohomology} and Theorem~\ref{thm:symp orbi cohomology}.

\begin{corollary}
$H_{T^r}^{*}(N;\Z)\simeq \Z[x_{1},x_{2},...,x_{m}]/\langle\I,
J_{1},J_{2},...,J_{k-r}\rangle$
\end{corollary}
\begin{proof}
The spectral sequence of the singular cohomology with $\Z$
coefficients of the fiber bundle
\begin{equation*}
N\hookrightarrow ET^r\times_{T^r}N\rightarrow BT^r
\end{equation*}
degenerates at the $E^{2}$ term, since both $N$ and $BT^r$ have
cohomology only in even degrees. Thus $ET^r\times_{T^r} N$ also has
cohomology only in even degrees.

Denote by $T^{k-r}$ the subgroup of $T^k$ that is the product of the
first $k-r$ copies of $S^1$ in $T^k= (S^1)^k$.

We complete the proof by applying the argument in Theorem 4.14 in
\cite{DJ} to the fiber bundle
\begin{equation*}
ET^r\times_{T^r}N \hookrightarrow
ET^{k-r}\times_{T^{k-r}}(ET^r\times_{T^r}N)\rightarrow BT^{k-r}
\end{equation*}
and also noticing that
$ET^{k-r}\times_{T^{k-r}}(ET^r\times_{T^r}N)=ET^k\times_{T^k}N$.
\end{proof}

\section{\bf The singular cohomology of good contact toric manifolds}\label{sec:cohomology contact}
Given a strictly convex good cone
\begin{equation}
 C= \bigcap_{i=1}^{m}\{x\in \R^n: \langle x, v_{i}\rangle \geq 0\},
\end{equation}
there is a good contact toric manifold $M$ associated to it. As we
proved in Proposition~\ref{prop:contact combina}, $M$ is
$T^n$-equivariantly homeomorphic to $P^{\lambda}$, which was defined
in Example~\ref{ex:contact combina}.

To make computations easier, we first move
$C_{0}=C\backslash\{\vec{0}\}$ into the upper half space $U\R^n=
\{(x_{1},x_{2},...,x_{n})\in \R^n:x_{n}>0\}$ using a transformation
in $SL(n;\Z)$, where $SL(n;\Z)$ is naturally included into
$SL(n;\R)$ as a subgroup. We will show in
Proposition~\ref{prop:moving to upper half} that we can always do
this. This will not change the homeomorphism type of the good
contact toric manifolds associated to the cone.

\begin{remark}
The good contact toric manifolds associated to two cones that differ
by a transformation in $SL(n;\Z)$ are contactomorphic. This fact was
not stated in \cite{L:contact toric}, but it easily follows from the
classification theorem there.
\end{remark}

\begin{lemma}\label{lemma:equiva of Assu A}
The following conditions are equivalent: \begin{enumerate}
\item\label{item: def of A} $C_{0}\subset U\R^n$,where $U\R^n$ stands for the `upper half space': $\{(x_{1},x_{2},...,x_{n})\in
\R^n:x_{n}>0\}.$
\item\label{item: dual cone} $(0,0,...,0,1)\in (C^{\vee})^{\circ}$, where $C^{\vee}=\{y\in \R^n: \langle y, x\rangle \geq 0, \forall x\in
C\}$, and $(C^{\vee})^{\circ}$ is its interior $\{y\in \R^n: \langle
y, x\rangle > 0, \forall x\in C_{0}\}$.
\item\label{item:linear combi} $(0,0,...,0,1)$ can be expressed as a linear combination of $\{v_{i},1\leq i\leq
m\}$ with positive coefficients.
\end{enumerate}
\end{lemma}
\begin{proof}
Denote $(0,0,...,0,1)$ by $\vec{a}$.

(\ref{item: def of A})$\Rightarrow$(\ref{item: dual cone}): For any
$x\in C_{0}, \langle \vec{a}, x\rangle =x_{n}
> 0$.

(\ref{item: dual cone})$\Rightarrow$(\ref{item:linear combi}): The
dual cone $C^{\vee}$ is spanned by rays along
$v_{1},v_{2},...,v_{m}$. So any vector in the interior of it can be
expressed as linear combination of $\{v_{i},1\leq i\leq m\}$ with
positive coefficients.

(\ref{item:linear combi})$\Rightarrow$(\ref{item: def of A}): For
any $x\in C$, we have $\langle x, v_{i}\rangle \geq 0$, for $1\leq
i\leq m$, with all equalities if and only if $x=\vec{0}$. Suppose
\begin{equation*}
\vec{a}=\sum_{i=1}^{m}k_{i}v_{i}, \mbox{with}\ k_{i}>0\ \forall
i.\end{equation*} Then \begin{equation*}x_{n}=\langle x,
\vec{a}\rangle= \sum_{i=1}^{m}k_{i}\langle x, v_{i}\rangle \geq
0.\end{equation*} The equality holds if and only if $x=\vec{0}$. So
for any $x\in C_{0}, x_{n}>0$, i.e., $C_{0}\subset U\R^n$.
\end{proof}

\begin{proposition}\label{prop:moving to upper half}
There is an element $B$ in $SL(n;\Z)$ so that $B(C_{0})$ is in
$U\R^n$.
\end{proposition}
\begin{proof}
Let $v=\sum_{i=1}^{m}v_{i}$ and $u$ be the primitive vector in the
direction of $v$. Suppose
\begin{equation*}
u=\frac{1}{k}v, \mbox{for}\ k\in \Z_{>0}.
\end{equation*}
Since $u$ is primitive, there exists $D\in SL(n,\Z)$, such that,
$D(u)=(0,0,...,0,1)$.

Define a transformation $B$ of $\R^n$ by:
\begin{equation}
x\mapsto Bx= D^{T}x ,
\end{equation}
where the matrix for $D^{T}$ is the transpose of the matrix for $D$.
Since $det D^{T}=det D=1$, we still have $B\in SL(n;\Z)$.

Under this transformation $B$, the cone $C = \bigcup_{i=1}^{m}\{x\in
\R^n: \langle x, v_{i}\rangle \geq 0\}$ is taken to
\begin{equation}B(C)= \bigcup_{i=1}^{m}\{x\in \R^n: \langle x,
D(v_{i})\rangle \geq 0\}.\end{equation}

So the normal vectors to the facets of the new cone $B(C)$ are
$\{D(v_{i}):1\leq i\leq m\}$. Moreover,
\begin{equation}
(0,0,...,0,1)=D(u)=\sum_{i=1}^{m}\frac{1}{k}D(v_{i}).
\end{equation}

Finally, using Lemma~\ref{lemma:equiva of Assu A}, we see that the
new cone $B(C_{0})$ is in $U\R^n$.
\end{proof}

So without loss of generality, we may now assume that $C_{0}\subset
U\R^n$. By intersecting $C$ with the hyperplane:
\begin{equation}
H= \{(x_{1},...,x_{n})\in \R^n: x_{n}=1\},
\end{equation}
we get a polytope, which we will denote by $P$.

The intersection of facet $F_{i}$ with $H$ is given by
\[\tilde{F_{i}}:=\left\{(x_{1},...,x_{n-1},1)\in \R^n:\sum_{j=1}^{n-1}x_{j}v_{ij}+v_{in} \geq 0\right\}.\]
Let $\F= \{\tilde{F_{i}}:1\leq i\leq m\}$.

 If we identify $H$ with $\R^{n-1}$ via
\[(x_{1},...,x_{n-1},1)\mapsto (x_{1},...,x_{n-1}),\]
 then $P$ can be thought
of as a polytope in $\R^{n-1}$ given by
\begin{equation}
 \{x\in
\R^{n-1}: \langle x,\tilde{v_{i}}\rangle \geq -v_{in}\},
\end{equation}
where $\tilde{v_{i}}=(v_{i1},...,v_{i,n-1})$ is normal to
$\tilde{F_{i}}$ in $H=\R^{n-1}$.

To compute the singular cohomology ring of $M$ with integer
coefficients, we need to add the following assumption.

 {\bf Smoothness Criterion}: The polytope $P$ is a Delzant polytope in
 $\R^{n-1}$ and $\tilde{v_{i}}$ is a primitive vector in $\Z^{n-1}$.

 \begin{remark}
This assumption allows us to stay in the smooth category in the
following discussion. In general, $P$ is just a simple rational
convex polytope, not necessarily smooth, so the
 results in this section do not hold in full generality. However, if
 we only care about rational cohomology, the argument and theorems in
 this section will still be valid in general. In that case, we need
 to use toric orbifolds instead of manifolds, and $S^1$ orbi-bundles instead of ``honest'' $S^1$ bundles.
 \end{remark}

Recall that there is a symplectic toric manifold $N$ associated with
$P$ which is equal to $P^{\tilde{\lambda}}$, where
$P^{\tilde{\lambda}}$ was defined in Example~\ref{ex:sym orb
combinatorics}, where we let $k=n-1$ and $\eta_{i}=-v_{in}$. The
symbol $\tilde{\Delta}$ will still be used here just as in
Example~\ref{ex:sym orb combinatorics}.

\begin{proposition}\label{S1 bundle}
With the notation as above, the good contact toric manifold $M$ is a
principal $S^{1}$ bundle over $N$. The projection map
\[d: M=(T^n\times P)/\Delta \rightarrow  N=(T^{n-1}\times P)/\widetilde{\Delta}\]
is given by
\[[t,p]\mapsto [\tilde{t},p],\]
where $t=(e^{i\theta_{1}},e^{i\theta_{2}},...,e^{i\theta_{n}})\in
T^n$, and
$\tilde{t}=(e^{i\theta_{1}},e^{i\theta_{2}},...,e^{i\theta_{n-1}})\in
T^{n-1}$ is just $t$ with the last coordinate dropped. The notation
$[t,p]$ and $[\tilde{t},p]$ refers to the equivalence classes of
$(t,p)$ and $(\tilde{t},p)$ respectively.
\end{proposition}
\begin{proof}
Let
\begin{equation}
S^1=\{(1,1,...,1,e^{i\theta_{n}})\in T^n: \theta_{n}\in \R\}
\end{equation}
act on $M=(T^n\times P)/\Delta$ by multiplication on $T^n$. Because
of the smoothness criterion we imposed, $S^1\cap \Delta_{p}=1$, for
any $p\in P$, and so this action is free.

The quotient space is
\begin{equation}
(T^n\times P)/\langle \Delta, S^1\rangle = (T^n\times P)/\langle
\widetilde{\Delta}, S^1 \rangle = (T^n/S^1\times
P)/\widetilde{\Delta} = (T^{n-1}\times P)/\widetilde{\Delta}.
\end{equation}
Therefore $N$ is the orbit space, and from the construction, we can
see the map from the space $M$ to $N$ just drops the last coordinate
of $t$.
\end{proof}

From this perspective, then, we get the Gysin sequence of the
$S^1$-bundle, namely the long exact sequence:
\[
\xymatrix{
   \ar[r]^{} & H^{2k-1}(N;\Z) \ar[r]^{\cup e} & H^{2k+1}(N;\Z) \ar[r]^{d^{*}} & H^{2k+1}(M;\Z) \ar[r]^{d_{*}} &
   H^{2k}(N;\Z)\ar`r[d]`[l] `[lllld]_{\cup e}`[dlll][dlll]
&\\
&H^{2k+2}(N;\Z) \ar[r]^{d^{*}} & H^{2k+2}(M;\Z) \ar[r]^{d_{*}} &
H^{2k+1}(N;\Z) \ar[r]^{\cup e}& H^{2k+3}(N;\Z) \ar[r] & }\]

The symbol $\cup e$ stands for the multiplication of the Euler class
of the $S^1$-bundle $d: M\rightarrow N$. The map $d^{*}$ is the
pull-back in cohomology rings by the fibration map $d: M\rightarrow
N$, so this is a ring homomorphism. The map $d_{*}$ is the map
commonly called {\em Gysin map}. It is not a ring homomorphism, but
it is a $H^{*}(N;\Z)$-module homomorphism in the following sense:
\begin{equation}\label{eq:push-pull}
d_{*}(d^{*}\alpha \cup \beta)= \alpha\cup d_{*}(\beta),
\end{equation}
for $\alpha\in H^{*}(N;\Z)$ and $\beta\in H^{*}(M;\Z)$.

Now the odd degree cohomology of $N$ vanishes and the long exact
sequence breaks down to short exact sequences
\begin{equation}\label{eq:ses}
\xymatrix{
  0 \ar[r]^{} & H^{2k+1}(M;\Z) \ar[r]^{d_{*}} & H^{2k}(N;\Z) \ar[r]^{\cup e} & H^{2k+2}(N;\Z) \ar[r]^{d^{*}} & H^{2k+2}(M;\Z) \ar[r]^{} & 0
  }.
\end{equation}

We will describe the chomology ring of $M$ in terms of the even part
and the odd part.
\begin{definition}\label{def:H_e}
Define the even part of $H^{*}(M;\Z)$ as
\begin{equation}H^{even}(M;\Z)=: \{\alpha\in H^{k}(M;\Z): k \ is\  an\
even\ integer\}.\end{equation}
 Define the odd part of $H^{*}(M;\Z)$ as
\begin{equation}H^{odd}(M;\Z)=: \{\alpha\in H^{k}(M;\Z): k\  is\  an\
odd\ integer\}.\end{equation}
\end{definition}

\begin{remark}\label{rmk:H_e}
It's easy to see that $H^{even}(M;\Z)$ is a subring of
$H^{*}(M;\Z)$, while $H^{odd}(M;\Z)$ is a module over
$H^{even}(M;\Z)$.
\end{remark}

Put together the exact sequences \eqref{eq:ses} of different
degrees, we get the following exact sequence:
\begin{equation}\label{eq:grouped ses}
\xymatrix{
  0 \ar[r]^{} & H^{odd}(M;\Z) \ar[r]^{d_{*}} & H^{*}(N;\Z) \ar[r]^{\cup e} & H^{*}(N;\Z)
\ar[r]^{d^{*}} & H^{even}(M;\Z) \ar[r]^{} & 0
  }.
\end{equation}
From this exact sequence, we can easily prove the following.

\begin{proposition}\label{prop:(co)ker rho}
Denote by $\rho$ the map
\[\cup e: H^{*}(N;\Z)\rightarrow H^{*}(N;\Z).\]
Then
\begin{equation}\label{eq:H_e}
H^{even}(M;\Z)\simeq coker\rho = H^{*}(N;\Z)/\langle e\rangle,
\end{equation}
where $\langle e\rangle$ stands for the ideal of $H^{*}(N;\Z)$
generated by the Euler class $e$. This is a ring isomorphism, and it
preserves the degree of the cohomology classes since it is induced
by the ring map $d^{*}$.

 Moreover,
\begin{equation}\label{eq:H_o}
H^{odd}(M;\Z)\simeq  ker\rho = Ann(e),
\end{equation}
where $Ann(e)$ denotes the annihilator of $e$ in $H^{*}(N;\Z)$. This
isomorphism maps a cohomology class in $H^{2k+1}(M;\Z)$ to a
cohomology class in $H^{2k}(N;\Z)$, lowering the degree by 1. It is
an $H^{even}(M;\Z)$-module isomorphism in the following sense.

The odd cohomology $H^{odd}(M;\Z)$ is an $H^{even}(M;\Z)$-module.
Moreover, $ker\rho = Ann(e)$ is an $coker\rho = H^{*}(N;\Z)/\langle
e\rangle$-module in the natural way, so it is also an
$H^{even}(M;\Z)$-module via the identification of $coker\rho$ and
$H^{even}(M;\Z)$ given in \eqref{eq:H_e}. The isomorphism in
\eqref{eq:H_o} respects these module structures.
\end{proposition}
\begin{proof}
The isomorphism \eqref{eq:H_e} follows from the exactness of
\eqref{eq:grouped ses} at the second $H^{*}(N;\Z)$ and
$H^{even}(M;\Z)$ and also the fact that $d^{*}:
H^{*}(N;\Z)\rightarrow H^{even}(M;\Z)$ is a ring homomorphism.

As for \eqref{eq:H_o}, first notice that it follows from
\eqref{eq:push-pull} that the map
\begin{equation}
d_{*}: H^{odd}(M;\Z)\rightarrow H^{*}(N;\Z)
\end{equation}
is a $H^{*}(N;\Z)$-module homomorphism. Exactness of
\eqref{eq:grouped ses} at $H^{odd}(M;\Z)$ and the first
$H^{2k}(N;\Z)$ shows that
\begin{equation}
d_{*}: H^{odd}(M;\Z)\rightarrow ker\rho
\end{equation}
is an $H^{*}(N;\Z)$-module isomorphism, hence an
$H^{*}(N;\Z)/\langle e\rangle$-module isomorphism.
\end{proof}

So to compute the cohomology of $M$, we only need to describe the
Euler class $e$ explicitly as a polynomial in
$x_{1},x_{2},...,x_{m}$, where $x_{i}$'s are generators of
$H^{*}(N;\Z)$, as in
 Theorem~\ref{thm:symp orbi cohomology}.

\begin{definition}
Define $\tilde{\pi_{\Z}}: \Z^m\rightarrow \Z^{n-1}$ by mapping the
$i^{th}$ standard basis vector $e_{i}$ of $\Z^m$ to $\tilde{v_{i}}$.
This induces a map
\begin{equation}
\tilde{\pi}: T^m\rightarrow T^{n-1}.
\end{equation}
in the same way we defined $\pi_{T}$ in \eqref{eq:pi_T}. Let
$\tilde{K}=ker \tilde{\pi}$. Since $P$ is Delzant and
$\tilde{v_{i}}$ is primitive, $\tilde{K}$ is connected, and there
exists a splitting $\tilde{r}: T^m\rightarrow \tilde{K}$ such that
$\tilde{r}|_{\tilde{K}}=id_{\tilde{K}}$.
\end{definition}

We will use the following diagram to compute the Euler class $e$.

 {
\[\xymatrix{
  M \ar[d]_{d}  & L_{1}=ET^{n-1}\times_{T^{n-1}}M \ar[d]_{d_{1}}  &
   L_{2} \ar[d]_{d_{2}}  & L_{3}=ET^m\times_{T^m}S^1 \ar[d]^{d_{3}} \\
  N \ar[r]^(.3){f_{1}} & ET^{n-1}\times_{T^{n-1}}N &
  (E\tilde{K}\times ET^{n-1})\times_{T^m}P^{\mu} \ar[l]_{f_{2}}  \ar[r]^(0.6){f_{3}} & ET^{m}\times_{T^m}pt  }\]
}

The base spaces and the maps are virtually the same as that of
\eqref{eq:diagram sym orb}. An element $t_m\in T^m$ acts on
$E\tilde{K}\times ET^{n-1}$ via the diagonal action of
$(\tilde{r}(t_m), \tilde{\pi}(t_m))$. We now turn to the top line.
On $L_{1}$, $T^{n-1}\circlearrowright M= (T^{n}\times P)/\Delta$ by
multiplication of $T^{n-1}$ on the first $n-1$ coordinates of
$T^{n}$.
 On $L_{3}$, $T^{m}\circlearrowright S^{1}$ with weight
$-(v_{1n},v_{2n},...,v_{mn})$. We will use $\vec{v^l}$ to denote the
vector $(v_{1n},v_{2n},...,v_{mn})$. Finally,
 $L_{2}$ is defined as the pull-back $f_{2}^{*}(L_{1})$.

\begin{lemma}\label{lemma:digram Euler}
As principal-$S^{1}$-bundles over $N$,
\begin{equation}\label{eq:f_1}
f_{1}^{*}(L_{1})=M.
\end{equation}
As principal $S^{1}$-bundles over $(E\tilde{K}\times
ET^{n-1})\times_{T^m}P^{\mu}$,
\begin{equation}\label{eq:L_2}
L_{2}=f_{2}^{*}L_{1}= f_{3}^{*}L_{3}.
\end{equation}
\end{lemma}

\begin{proof}
The map $f_{1}$ is an inclusion, so the pull back of $L_{1}$ by
$f_{1}$ is just by restriction. Then the equation \eqref{eq:f_1}
follows from the definition.

To show \eqref{eq:L_2}, it's enough if we can construct an
$S^1$-equivariant map from $f_{3}^{*}L_{3}$ to $f_{2}^{*}L_{1}$ ,
which lifts the identity map of the base space $(E\tilde{K}\times
ET^{n-1})\times_{T^m}P^{\mu}$.

Let $q= [x,y,[e^{i\vec{\theta}},p]]$ be a point in the base space
$(E\tilde{K}\times ET^{n-1})\times_{T^m}P^{\mu}$, where $x\in
E\tilde{K}, y\in ET^{n-1}, \vec{\theta}=(\theta_{1},...,\theta_{m}),
e^{i\vec{\theta}}=(e^{i\theta_{1}},...,e^{i\theta_{m}})\in T^m$, and
$p\in P$.

The fiber of $f_{3}^{*}L_{3}$ over the point $q$ is, by the
definition of pull-back, the fiber of $L_{3}$ over the point
$f_{3}(q)=[(x,y),pt]$, namely
\begin{equation}\label{eq:1}
f_{3}^{*}L_{3}|_{q}=\{[(x,y), e^{i\beta}]:\beta\in \R\}.
\end{equation}

 The fiber of
$f_{2}^{*}L_{1}$ over the point $q$ is, by the definition of
pull-back, the fiber of $L_{1}$ over the point $f_{2}(q)=[y,[
\tilde{\pi}(e^{i\vec{\theta}}),p]]$, namely
\begin{equation}\label{eq:2}
f_{2}^{*}L_{1}|_{q}=\{[y,[(\tilde{\pi}(e^{i\vec{\theta}}),e^{i\alpha}),p]]:
\alpha\in \R\}.
\end{equation}
Define a map from $f_{3}^{*}L_{3}|_{q}$ to $f_{2}^{*}L_{1}|_{q}$ by
sending
\begin{equation}\label{eq:3}
[x,y, e^{i\beta}]\mapsto
[y,[(\tilde{\pi}(e^{i\vec{\theta}}),e^{i\beta+i\langle
\vec{v^{l}},\vec{\theta}\rangle}),p]].
\end{equation}
We call this map $s_{q}$. To show $s_{q}$ is well-defined, we need
to show the definition is independent of the choice of
representative of $q$.

First, assume $e^{i\vec{\phi}}\in T^m$, and so
\begin{equation}
[x,y,[e^{i\vec{\theta}},p]]=[\tilde{r}(e^{i\vec{\phi}})x,\tilde{\pi}(e^{i\vec{\phi}})y,[e^{i(\vec{\theta}+\vec{\phi})},p]].
\end{equation}
Using this representative of $q$, then
\begin{eqnarray*}\label{eq:4}
s_{q}([x,y,e^{i\beta}])& = &
s_{q}([\tilde{r}(e^{i\vec{\phi}})x,\tilde{\pi}(e^{i\vec{\phi}})y,
e^{i(\beta+\langle -\vec{v^{l}},\vec{\phi}\rangle)}])\\
& = &
[\tilde{\pi}(e^{i\vec{\phi}})y,[(\tilde{\pi}(e^{i(\vec{\theta}+\vec{\phi})}),e^{i\beta-i\langle
\vec{v^{l}},\vec{\phi}\rangle
+i\langle\vec{v^{l}},(\vec{\theta}+\vec{\phi}) \rangle}),p]]\\
& = &
[\tilde{\pi}(e^{i\vec{\phi}})y,[(\tilde{\pi}(e^{i\vec{\phi}})\tilde{\pi}(e^{i\vec{\theta}}),e^{i\beta+i\langle
 \vec{v^{l}},\vec{\theta}\rangle}),p]].
\end{eqnarray*}
This equals to the
 RHS of \eqref{eq:3}
by the definition of $L_{1}$ as given before Lemma~\ref{lemma:digram
Euler}.

Second, choose a different representative by letting
$e^{i\vec{\delta}}\in \Omega_{p}$, and then
\begin{equation}
[x,y,[e^{i\vec{\theta}},p]]=[x,y,[e^{i\vec{\delta}}e^{i\vec{\theta}},p]].
\end{equation}
Using this representative of $q$, we have:
\begin{eqnarray*}\label{eq:5}
s_{q}([(x,y),e^{i\beta}]) & = &
[y,[(\tilde{\pi}(e^{i\vec{\delta}}e^{i\vec{\theta}}),e^{i\beta+i\langle
\vec{v^{l}},\vec{\delta}+\vec{\theta}\rangle}),p]]\\
& = &
[y,[(\tilde{\pi}(e^{i\vec{\delta}})\cdot\tilde{\pi}(e^{i\vec{\theta}}),e^{i\langle\vec{v^l},\vec{\delta}\rangle}\cdot
e^{i\beta+i\langle\vec{v^l},\vec{\theta}\rangle}),p]].
\end{eqnarray*}
This also equals to the RHS of \eqref{eq:3} since
\begin{equation}\label{eq:6}
(\tilde{\pi}(e^{i\vec{\delta}}),e^{i\langle\vec{v^l},\vec{\delta}\rangle})\in
\Delta_{p},
\end{equation}
which is true by definition of $\Delta, \widetilde{\Delta}$ and
$\vec{v^l}$.

So the map $s_{q}$ defined by \eqref{eq:3} is well-defined. It's
obviously $S^{1}$-equivariant.
\end{proof}

Since the Euler class of $L_{3}$ is $\sum_{i=1}^{m}v_{in}x_{i}$,
using Lemma~\ref{lemma:digram Euler} and  the naturality of Euler
classes, we conclude
\begin{equation}\label{eq:e}
e=\sum_{i=1}^{m}v_{in}x_{i}.
\end{equation}

\begin{definition}
We define linear forms
\begin{equation}
J_{k}=\sum_{i=1}^{m}v_{ik}x_{i}, 1\leq k\leq n,
\end{equation}
and the ideals
\begin{equation}
\J = \langle J_{1},\cdots,J_{n-1},J_{n}\rangle
\end{equation}
and
\begin{equation}
\widetilde{\J}=\langle J_{1},\cdots,J_{n-1}\rangle,
\end{equation}
where $\langle S \rangle$ denotes the ideal in $\Z[x_{1},...,x_{m}]$
generated by the elements of $S$.
\end{definition}
\begin{remark}
Notice that by definition we have $e=J_{n}$.
\end{remark}

Combining Theorem~\ref{thm:symp orbi cohomology},
Proposition~\ref{prop:(co)ker rho} and \eqref{eq:e}, we have proved
the following:
\begin{theorem}\label{thm:contact cohomology}
Assume \begin{equation}
 C= \bigcap_{i=1}^{m}\{x\in \R^n: \langle x, v_{i}\rangle \geq 0\}
\end{equation} is a strictly convex good cone and $M$ is the good contact toric
manifold associated with it. Further assume that
$C_{0}=C\backslash\{\vec{0}\}\subset U\R^n$ and the smoothness
criterion for $C$ holds. Let
\begin{equation}\rho:\frac{\Z[x_{1},...,x_{m}]}{\langle \mathcal
{I},\widetilde{\J} \rangle}\rightarrow
\frac{\Z[x_{1},...,x_{m}]}{\langle \mathcal
{I},\widetilde{\J}\rangle}
\end{equation}
be multiplication by $J_{n}$. Then:

\begin{equation}\label{eq:7}
H^{even}(M;\Z)\simeq coker\rho\simeq \Z[x_{1},...,x_{m}]/\langle
\I,\J \rangle\end{equation} as rings, where $x_{i}$ represents a
cohomology class of degree two.

 Moreover,
\begin{equation}\label{eq:8}
H^{odd}(M;\Z)\simeq ker\rho \end{equation} as $(H^{even}(M;\Z)\simeq
coker\rho)$-modules. A homogeneous polynomial of degree $k$
represents a cohomology class of degree $2k+1$ under this
isomorphism.
\end{theorem}

The next theorem states that half of the Betti numbers of $M$
vanish.

\begin{theorem}\label{thm:vanishing odd}Under the same assumption
for  $M$ as in Theorem~\ref{thm:contact cohomology}, we have
\begin{equation}H^{2k+1}(M;\Z)=0
\end{equation}
for $\{k\in \N: 1\leq 2k+1\leq n-1\}$, and
\begin{equation}
H^{2k}(M;\Q)=0
\end{equation}
for $\{k\in \N:n\leq 2k\leq 2n-2\}$.
\end{theorem}
\begin{proof}
According to Theorem 6.3 in \cite{Guillemin:Kaehler form}, the
cohomology class of the K$\ddot{a}$hler form on $N$ is given by
\begin{equation}
[\omega]=2\pi\sum_{i=1}^{m}v_{in}D_{i},
\end{equation}
where $D_{i}$ denotes the cohomology class in $H^{2}(N;\R)$ that is
the Poincare dual to $(T^{n-1}\times \tilde{F_{i}})/\tilde{\Delta}$.
According to Proposition~\ref{prop:Poincare dual} in the Appendix,
$D_{i}=-x_{i}$. So
\begin{equation}
[\omega]=-2\pi e.
\end{equation}
Using the Hard Lefschetz Theorem on $N$, it follows easily from
Theorem~\ref{thm:contact cohomology} that
\begin{equation}H^{2k+1}(M;\Q)=0
\end{equation}
for $\{k\in \N: 1\leq 2k+1\leq n-1\}$, and
\begin{equation}
H^{2k}(M;\Q)=0
\end{equation}
for $\{k\in \N:n\leq 2k\leq 2n-2\}$. Furthermore, as the cohomology
ring of a symplectic toric manifold, the ring
$\Z[x_{1},x_{2},...,x_{m}]/\{\I,\tilde{\J}\}$ in
Theorem~\ref{thm:contact cohomology} is torsion-free (see
\cite{toric variety}). Therefore as an ideal of it, $ker\rho$ is
also torsion-free. Hence
\begin{equation}H^{2k+1}(M;\Z)=0
\end{equation}
for $\{k\in \N: 1\leq 2k+1\leq n-1\}$.
\end{proof}

\begin{corollary}\label{cor:trivial between odd}
Under the same assumption for $M$ as in Theorem~\ref{thm:contact
cohomology}, the product of two odd-degree cohomology classes of $M$
is zero.
\end{corollary}
\begin{proof}
This is because of dimension considerations.
\end{proof}

\begin{remark}
Theoretically, Theorem~\ref{thm:contact cohomology} already tells us
all the information of $H^{*}(M;\Z)$ as an additive group. Perhaps
there should be a combinatorial proof for Theorem~\ref{thm:vanishing
odd}, but the author could not find one.
\end{remark}

\begin{remark}
Theorem~\ref{thm:contact cohomology} tells us what the elements of
$H^{*}(M;\Z)$ are, \eqref{eq:7} tells us how to multiply two
even-degree cohomology classes, \eqref{eq:8} tells us how to
multiply an even-degree cohomology class and an odd-degree
cohomology class. Corollary~\ref{cor:trivial between odd} tells us
the product of two odd-degree cohomology classes must be zero. So
the ring structure of $H^{*}(M;\Z)$ is now completely determined.
\end{remark}

\section{\bf Appendix: Several equivalent descriptions of the
generators of cohomology rings of symplectic toric
manifolds}\label{sec:comparison}

There is a wide variety of descriptions of the generators of the
cohomology ring of a symplectic toric manifold in the literature. In
this appendix we list some of them and discuss their relations.

Let $P$ be as in \eqref{eq:simple polytope},
\begin{equation}
P=\bigcap_{i=1}^{m}\{x\in \R^k: \langle x, \tilde{v_{i}}\rangle \geq
\eta_{i}\}.
\end{equation}
We will use $N$ to denote the symplectic toric manifold associated
to $P$. We will keep our notation from Example~\ref{ex:sym orb
combinatorics}, so $N= P^{\tilde{\lambda}}= (T^k\times
P)/\tilde{\Delta}$.

The cohomology in this section are assumed to be integral
cohomology.

 The generators in Theorem~\ref{thm:symp orbi cohomology}
of this paper, which we denoted by $x_{i}$, is the image of the
generators of $H^{*}(BT^m)$ under the composition:

{\small \begin{equation}H^{*}(BT^m)\rightarrow
H^{*}(ET^m\times_{T^m}((T^m\times
P)/\Omega))\stackrel{\sim}{\rightarrow}H^{*}(ET^k\times_{T^k}((T^k\times
P)/\tilde{\Delta}))\rightarrow H^{*}(N)
\end{equation}}

 We will now give an easier description for these generators,
starting by defining a $S^1$-bundle over $N=(T^k\times
P)/\tilde{\Delta}$.

For any fixed $j\in\{1,2,...,k\}$, define a characteristic map from
the set of facets of $P$ to $\Z^{k+1}$:
\begin{eqnarray*}
\sigma_{j}: & \F &\rightarrow \Z^{k+1}\\
        & \tilde{F_{j}}&\mapsto (v_{j1},v_{j2},...,v_{jk},1);\\
        & \tilde{F_{i}}&\mapsto (v_{i1},v_{i2},...,v_{ik},0)\ \
        \mbox{for}\  i\neq j.
\end{eqnarray*}
Then $P^{\sigma_{j}}$ is a principal $S^1$-bundle over
$P^{\tilde{\lambda}}$ for similar reasons as in Proposition~\ref{S1
bundle}.

For the same reason as \eqref{eq:L_2} and \eqref{eq:e}, we have the
following proposition.

\begin{proposition} \label{prop:x_i is chern class of a artificial
bundle} Each generator $x_{j}$ in Theorem~\ref{thm:symp orbi
cohomology} is the Euler class, or equivalently, the first Chern
class, of the principal $S^1$-bundle $P^{\sigma_{j}}$ over
$N=P^{\tilde{\lambda}}$.
\end{proposition}

Recall from \eqref{eq:toric manifold as quotient}, $N$ is the
quotient of a set $(i^*\circ u)^{-1}(i^{*}(-\eta))$, which we will
denote by $Z$, by a torus $K$. Then $Z$ is a principal-$K$ bundle
over $N$.

In the book \cite{book:equivariant deRham}, the cohomology of a
symplectic toric manifold is computed in a very different way
(Theorem 9.8.6 of the book). The generators for the cohomology ring
are $c_{1},c_{2},...,c_{m}$, where $(c_{1},c_{2},...,c_{m})$ are the
Chern classes of the bundle $Z\rightarrow N$. These $c_{i}$'s can
also be described in the following way, as illustrated in both
Section 2.2 of \cite{Guillemin:book} and Section 9.8 of
\cite{book:equivariant deRham}.

The torus $T^m$ acts on $\C^m$ in the standard way and $\C^{m}$
splits
\begin{equation}
\C^m = \C_{1}\oplus \C_{2}\oplus\cdots \oplus \C_{m}.
\end{equation}
The torus $K$, as a subgroup of $T^m$, also acts on $\C^m$ and
preserves this splitting. Then
\begin{equation}
(Z\times \C_{i})/K\rightarrow Z/K
\end{equation}
is a complex line bundle over $N=Z/K$, where the action of $K$ on
$\C_{i}$ is by first including $K$ into $T^m$, then acting with
weight $-(0,...,0,1,0,...,0)$, where the only $1$ is on the $i^{th}$
position. The $K$ action on $Z\times \C_{i}$ is the diagonal action.
Then the first Chern class of this line bundle is $c_{i}$.

\begin{proposition}\label{prop:chern class of torus bundle}
The generator $c_{i}$ is exactly the generator $x_{i}$ used in
Theorem~\ref{thm:symp orbi cohomology}.
\end{proposition}
\begin{proof}
Assume without loss of generality that $i=1$. According to
Proposition~\ref{prop:x_i is chern class of a artificial bundle}, it
suffices to show $P^{\sigma_{1}}$ and $(Z\times S^1)/K$ are
isomorphic principal-$S^1$ bundles over $P^{\tilde{\lambda}}=Z/K$,
where the action of $K$ on $S^1$ is by first including $K$ into
$T^m$, then acting with weight $-(1,0,0,...,0)$. The base spaces are
identified by Proposition~\ref{prop: DJ construction is OK for sym
orb}.

Fixing a splitting $\alpha : T^k\rightarrow T^m$, such that
\begin{equation}
\pi\circ\alpha=id_{T^k}.
\end{equation}
Define
\begin{eqnarray*}
f: & P^{\sigma_{1}} & \rightarrow (Z\times S^1)/K \\
   & [t_{k},e^{i\theta},p] & \mapsto [\alpha(t_{k}).u_{0}^{-1}(\pi^{*}(p)-\eta),
   \alpha(t_{k}).e^{i\theta}],
\end{eqnarray*}
where $t_{k}\in T^k, \theta\in \R, p\in P$, and $u_{0}: (\R_{\geq
0})^{m}\rightarrow (\R_{\geq 0})^{m}$ is given by
$(z_1,\cdots,z_m)\mapsto (z_1^2,\cdots, z_m^2)$.  The verification
that this map is well-defined is routine. It is easy to see this map
lifts the identity map of the base space, is $S^1$-equivariant and
non-trivial on each fiber. So $P^{\sigma_{1}}$ and $(Z\times S^1)/K$
are isomorphic, hence $c_{1}=x_{1}$.
\end{proof}

The set $(T^k\times\tilde{F_{i}})/\tilde{\Delta}$ is a submanifold
of $N$. It is the pre-image of $\tilde{F_{i}}$ under the moment map
$\nu$ as defined in \eqref{eq:moment map}. Its Poincare dual, by
definition, is a cohomology class in $H^{2}(N;\R)$. We denote it by
$D_{i}$.

\begin{proposition}\label{prop:Poincare dual}
The class $D_{i}$ equals $-c_{i}=-x_{i}$.
\end{proposition}
\begin{proof}
Without loss of generality, assume $i=1$. It is obvious that
$-c_{1}$ is the first Chern class of the complex line bundle
\begin{equation}
(Z\times \C)/K\rightarrow Z/K,
\end{equation}
where $K$, as a subgroup of $T^m$, acts on $\C$ with weight
$(1,0,0,...,0)$. According to Proposition 12.8 in \cite{BT:book}, it
suffices to find a transversal section of this line bundle, such
that the zero locus of the section is exactly
$\nu^{-1}(\tilde{F_{1}})$. Define a section
\begin{eqnarray*}
s: &Z/K & \rightarrow (Z\times \C)/K\\
&   [z_{1},z_{2},...,z_{m}]&\mapsto [(z_{1},z_{2},...,z_{m}),z_{1}],
\end{eqnarray*}
noticing that $Z$ is a subset of $\C^m$. It is easy to see this map
is well-defined, and it is straightforward to show that
\begin{equation}
z_{1}=0 \Leftrightarrow \nu([z_{1},...,z_{m}])\in \tilde{F_{1}}.
\end{equation}
Finally, to see it is transversal to the zero section, simply notice
it is holomorphic and obviously not tangent to the zero section
along the zero locus.
\end{proof}

In \cite{TW:quotients}, as a corollary of a more general theorem,
there is yet another way of computing the cohomology ring of a
symplectic toric manifold. To describe the generators there, we draw
a diagram first:
\begin{equation}
\xymatrix{
  (EK\times ET^k)\times_{T^m}\C^m & EK\times_{K}\C^m \ar[l]_(0.4){g_{1}} & EK\times_{K}Z \ar[l]_{g_{2}} \ar[r]^(0.55){g_{3}}& Z/K
  }.
\end{equation}
In the diagram, $g_{1}$ and $g_{2}$ are just inclusion maps, $g_{3}$
is a fiber bundle with fiber $EK$. The group $T^m$ acts on $EK\times
ET^k$ as explained in Section~\ref{sec:symplectic orbifold}. Now
\begin{equation}H^{*}((EK\times ET^k)\times_{T^m}\C^m) =
\Z[y_{1},...,y_{m}],\end{equation} where $y_{i}$ is the Chern class
of the principal $S^1$-bundle
\begin{equation}
(EK\times ET^k)\times_{T^m}(\C^m\times S^1)\rightarrow (EK\times
ET^k)\times_{T^m}\C^m,
\end{equation}
where $T^m$ acts on $S^1$ with weight $-(0,0,...,0,1,0,...,0)$,
where the only $1$ is on the $i^{th}$ position. In Theorem 7 of
\cite{TW:quotients}, the generators of the cohomology ring of a
symplectic toric manifold are the images of these $y_{i}$'s under
the composed map
\begin{equation}\small
\xymatrix{ H^{*}((EK\times ET^k)\times_{T^m}\C^m)
\ar[r]^(0.6){g_{1}^{*}} & H^{*}(EK\times_{K}\C^m) \ar[r]^{g_{2}^{*}}
& H^{*}(EK\times_{K}Z) \ar[r]^(0.6){(g_{3}^{*})^{-1}} & H^{*}(Z/K)
}.
\end{equation}

It follows easily from the naturality of Chern class that
$g_{2}^{*}g_{1}^{*}(y_{i})=g_{3}^{*}(c_{i})$, whence we may conclude
our final proposition.

\begin{proposition}
The generators for the cohomology ring of a symplectic toric
manifold in Theorem 7 in \cite{TW:quotients} are the same as the
ones used in Theorem~\ref{thm:symp orbi cohomology} in this paper.
\end{proposition}

\bibliographystyle{amsalpha}

\end{document}